\documentclass[11pt]{article}

\usepackage[margin=1in]{geometry}
\usepackage[T1]{fontenc}
\usepackage[utf8]{inputenc}
\usepackage{lmodern}
\usepackage{microtype}
\usepackage{amsmath,amssymb,amsthm,mathtools}
\usepackage{enumitem}
\usepackage{tikz}
\usetikzlibrary{positioning,fit,calc,shapes.geometric}
\usepackage[colorlinks=true,linkcolor=blue,citecolor=blue,urlcolor=blue]{hyperref}
\hypersetup{
  pdftitle={Hypergraph Universality and Erd\H{o}s--P\'{o}sa Failure for Chromatically Rich Odd Cycles},
  pdfauthor={Shuyan Chen},
  pdfkeywords={odd cycle, chromatic number, Erdos-Posa property, clutter, hypergraph transversal, finite-congestion packing, fractional packing}
}

\newtheorem{theorem}{Theorem}[section]
\newtheorem{lemma}[theorem]{Lemma}
\newtheorem{proposition}[theorem]{Proposition}
\newtheorem{corollary}[theorem]{Corollary}
\theoremstyle{definition}
\newtheorem{definition}[theorem]{Definition}
\newtheorem{problem}[theorem]{Problem}
\theoremstyle{remark}
\newtheorem{remark}[theorem]{Remark}

\newcommand{\calC}{\mathcal C}

\newcommand{\calF}{\mathcal F}
\newcommand{\calH}{\mathcal H}

\newcommand{\calX}{\mathcal X}
\newcommand{\tr}{\operatorname{tr}}

\title{Hypergraph Universality and Erd\H{o}s--P\'{o}sa Failure for Chromatically Rich Odd Cycles}
\author{Shuyan Chen\\
\small Department of Mathematics, University of Manchester\\
\small \texttt{shuyan.chen-2@student.manchester.ac.uk}}
\date{}

\begin{document}
\maketitle

\begin{abstract}
For every fixed $k\ge3$, every finite nonempty clutter $\calH$ can be realized exactly as the family of inclusion-minimal terminal traces of the \emph{$k$-bad} odd cycles, namely the odd cycles $C$ satisfying
\[
\chi(G[V(C)])\ge k+1.
\]
The host graph $G$ may be chosen $2$-connected, with $T=V(\calH)$ independent, and with
\[
\chi(G)=k+1,\qquad \omega(G)=k,\qquad |E(G)|\le k|V(G)|.
\]
Given any integer $B\ge1$, one realization simultaneously preserves the transversal number, the fractional packing number, and every integer $c$-packing number for $1\le c\le B$.  Thus chromatic richness on odd-cycle spans has the full packing-covering complexity of arbitrary finite set systems, even in sparse graphs whose chromatic number exceeds their clique number by one.

As a consequence, for every $b,N\ge1$ and every $\varepsilon>0$, there is such a graph with
\[
\tau_k(G)=N,\qquad
\nu_k^{1/c}(G)=c\quad(1\le c\le b),\qquad
\nu_k^*(G)<1+\varepsilon.
\]
Hence every finite-congestion and fractional Erd\H{o}s--P\'{o}sa property fails for $k\ge3$, with an asymptotically optimal fractional obstruction.  This contrasts sharply with $k=2$, where the relevant cycles are the ordinary odd cycles: the integral property fails, while Reed's theorem yields every congestion level at least two and the fractional property.  We also characterize the shortest $k$-bad odd cycles by a single fixed witness graph.
\end{abstract}



\section{Introduction}
Chromatic richness on the vertex set of a single odd cycle is sufficiently expressive to encode arbitrary finite set systems.  The main result of this paper shows that, from the first nonclassical threshold onward, every finite clutter occurs exactly as the minimal terminal-trace system of such cycles, with its transversal and packing parameters preserved.  This universality yields a sharp transition in Erd\H{o}s--P\'{o}sa behaviour.

Following a problem of Erd\H{o}s and Hajnal recorded by Gy\'{a}rf\'{a}s~\cite{GyarfasMemories}, define
\[
\psi(G)=\max\{\chi(G[V(C)]):C\text{ is an odd cycle of }G\},
\]
and put $\psi(G)=0$ if $G$ is bipartite.  For a fixed integer $k\ge2$, let $\calC_k(G)$ be the family of odd cycles $C$ satisfying
\[
\chi(G[V(C)])\ge k+1;
\]
we call them \emph{$k$-bad}.  We study the integral, bounded-congestion, and fractional packing-covering theory of $\calC_k(G)$.  A companion preprint studies the finite-order extremal separation between $\chi(G)$ and $\psi(G)$~\cite{ChenSpanDefect}; the present paper studies the individual cycles that witness large chromatic span.

The classical Erd\H{o}s--P\'{o}sa theorem initiated the systematic study of packing-covering dualities for cycles~\cite{ErdosPosa}; see~\cite{RaymondThilikos} for a survey.  Odd cycles fail the integral Erd\H{o}s--P\'{o}sa property~\cite{DejterNeumannLara}, whereas Reed proved their half-integral property~\cite{Reed}.  Prescribed vertices, parity, modularity, and group labels lead to a broad theory of constrained cycles~\cite{KakimuraKawarabayashiMarx,PontecorviWollan,KakimuraKawarabayashiMod,KakimuraKawarabayashiOdd,HuynhJoosWollan,ThomasYoo,GollinEtAlHalf,GollinEtAlFull}; in particular, modern group-labelled frameworks provide strong packing-covering theorems for broad classes governed by finitely many fixed constraints.  Gorsky, Hendrey, and Huynh proved failure at every finite congestion level for prime-length cycles and, more generally, for cycles whose lengths belong to a set of lower density zero, even in planar graphs~\cite{GorskyHendreyHuynh}.  Their theorem reveals strong obstructions for prescribed length families.  The phenomenon here is universal set-system complexity: for every $k\ge3$, chromatically rich odd cycles realize every finite clutter exactly, while preserving the transversal and fractional packing optima and any prescribed finite initial segment of the congestion hierarchy.

To state the main result, let $T\subseteq V(G)$.  The \emph{$T$-trace} of a cycle $C$ is $V(C)\cap T$.  A clutter is a hypergraph in which no edge properly contains another.  We take every nonempty clutter on its support
\[
V(\calH)=\bigcup_{e\in E(\calH)}e.
\]

\begin{theorem}[Hypergraph universality, informal]\label{thm:intro-universality}
Fix $k\ge3$.  Every finite nonempty clutter $\calH$, with terminal set $T=V(\calH)$, can be realized as the family of inclusion-minimal $T$-traces of the $k$-bad odd cycles in a $2$-connected graph $G$ containing $T$ as an independent set and satisfying
\[
\chi(G)=k+1,\qquad \omega(G)=k,\qquad |E(G)|\le k|V(G)|.
\]
Moreover, for every prescribed integer $B\ge1$, the construction can be chosen so that one graph $G$ preserves exactly the transversal number, the fractional packing number, and all integer $c$-packing numbers for $1\le c\le B$.
\end{theorem}

The key mechanism is the converse colouring property
\[
X\cap T\text{ contains no edge of }\calH
\quad\Longrightarrow\quad
G[X]\text{ is }k\text{-colourable}.
\]
It controls every bad cycle in the host graph, including cycles assembled across several gadgets.  A one-vertex-repairable colour-forcing core realizes each hyperedge, a private odd frame converts the core into a bad cycle, and private amplification gives the exact integral and fractional parameter identities.  Together these ingredients yield an exact packing-covering transfer from finite clutters to chromatically rich odd cycles.  In particular, universal set-system complexity already occurs in sparse, $2$-connected graphs with $\chi(G)-\omega(G)=1$.

Applying the realization theorem to complete uniform hypergraphs gives the sharp global consequence.

\begin{theorem}[Simultaneous bounded-congestion and fractional failure]\label{thm:intro-sharp-failure}
Let $k\ge3$.  For all integers $b,N\ge1$ and every $\varepsilon>0$, there is a $2$-connected graph $G$ such that
\[
\chi(G)=k+1,\qquad \omega(G)=k,\qquad |E(G)|\le k|V(G)|,
\]
and
\[
\tau_k(G)=N,\qquad
\nu_k^{1/c}(G)=c\quad\text{for every }1\le c\le b,
\qquad
\nu_k^*(G)<1+\varepsilon.
\]
\end{theorem}

The value $1+\varepsilon$ is asymptotically best possible: a nonempty finite set system with transversal number larger than one has fractional packing number strictly larger than one.  At $k=2$, the $k$-bad cycles are precisely the odd cycles.  A half-integral packing of $h$ cycles gives a fractional packing of value $h/2$, so
\[
\nu^{1/2}(G)\le2\nu^*(G),
\]
and $\nu^{1/2}(G)\le\nu^{1/b}(G)$ for every $b\ge2$.  The results of Dejter--Neumann-Lara and Reed therefore complete the other side of the transition:
\[
\begin{array}{ll}
k=2:
& \text{the ordinary integral property fails, whereas every congestion}\\[-1pt]
& \text{level $b\ge2$ and the fractional version hold};\\[2mm]
k\ge3:
& \text{every finite congestion level $b\ge1$ and the fractional}\\[-1pt]
& \text{version fail}.
\end{array}
\]

The shortest bad cycles form a rigid companion to unrestricted universality.  Put
\[
\ell_k=
\begin{cases}
 k+1,&k\text{ even},\\
 k+2,&k\text{ odd}.
\end{cases}
\]
For even $k$, let $W_k=K_{k+1}$.  For odd $k$, obtain $W_k$ from $K_{k+1}$ by adding one vertex adjacent to two clique vertices.  We prove that an odd cycle of the minimum possible length $\ell_k$ is $k$-bad exactly when its span contains $W_k$.  Every $W_k$ contains a $K_{k+1}$, whereas the universal realization graphs have clique number $k$; the arbitrary-clutter phenomenon therefore occurs entirely beyond the shortest layer.

Section~\ref{sec:parameters} fixes the packing notation.  Section~\ref{sec:realization} proves the realization theorem, Section~\ref{sec:sharp} derives its sharp consequences, and Section~\ref{sec:shortest} identifies the shortest witness.  Section~\ref{sec:bounded} records the bounded-certificate outlook, reducing it to finitely many parity-linkage families with a moving real copy of a fixed critical core.

\section{Packing parameters}\label{sec:parameters}

All graphs are finite and simple, and all hypergraphs are finite with nonempty hyperedges.  For a nonempty hypergraph $\calH$, its \emph{support} is
\[
V(\calH)=\bigcup_{e\in E(\calH)}e.
\]
Throughout, clutters are taken on their support.  Removing isolated ground elements does not change $\tau(\calH)$, $\nu_b^{\mathbb Z}(\calH)$, or $\nu^*(\calH)$.  For a family $\calF$ of nonempty subsets of a ground set $V$, write
\[
\tau(\calF)=\min\{|X|:X\subseteq V,\ X\cap F\ne\emptyset\text{ for every }F\in\calF\}.
\]
For $b\in\mathbb N$, define the integer $b$-packing number
\[
\nu_b^{\mathbb Z}(\calF)
 =\max\left\{\sum_{F\in\calF}z_F:
 z_F\in\mathbb Z_{\ge0},\ 
 \sum_{F\ni v}z_F\le b\ \text{for every }v\in V
 \right\}.
\]
Thus parallel use of the same incidence set is allowed by the integer variable $z_F$.  This is the correct trace parameter when a graph contains several distinct cycles with the same terminal trace.  For $b=1$, it is the ordinary matching number.

The fractional packing number is
\[
\nu^*(\calF)
 =\max\left\{\sum_{F\in\calF}y_F:
 y_F\ge0,\ 
 \sum_{F\ni v}y_F\le1\ \text{for every }v\in V
 \right\}.
\]
By linear-programming duality, this equals the minimum weight of a fractional transversal.

For a family $\calC$ of cycles, let $\nu^{1/b}(\calC)$ be the maximum cardinality of a set of distinct members of $\calC$ in which every vertex occurs at most $b$ times, and put $\nu(\calC)=\nu^{1/1}(\calC)$.  Let $\nu^*(\calC)$ denote the fractional packing number of the cycle-vertex hypergraph.

For the bad-cycle family $\calC_k(G)$, we abbreviate
\[
\tau_k(G)=\tau\bigl(\{V(C):C\in\calC_k(G)\}\bigr)
\]
and write $\nu_k^{1/b}(G)$ for the maximum cardinality of a set of \emph{distinct} cycles from $\calC_k(G)$ in which every vertex occurs at most $b$ times.  We write $\nu_k^*(G)$ for the fractional packing number of the cycle-vertex hypergraph.  The notation agrees with the usual integral case $b=1$ and the half-integral case $b=2$.

For fixed $k$ and $b$, we say that the $k$-bad odd cycles have the \emph{$1/b$-integral Erd\H{o}s--P\'{o}sa property} if there is a nondecreasing function $f\colon\mathbb N_0\to\mathbb N_0$ such that
\[
\tau_k(G)\le f\bigl(\nu_k^{1/b}(G)\bigr)
\]
for every graph $G$.  They have the \emph{fractional Erd\H{o}s--P\'{o}sa property} if there is a nondecreasing function $f\colon\mathbb N_0\to\mathbb N_0$ such that
\[
\tau_k(G)\le f\bigl(\lceil\nu_k^*(G)\rceil\bigr)
\]
for every graph $G$.

If $T\subseteq V(G)$ and $C$ is a cycle, set
\[
\tr_T(C)=V(C)\cap T.
\]
For a family of cycles $\calC$, let
\[
\tr_T(\calC)=\{\tr_T(C):C\in\calC\}.
\]

\section{The hypergraph realization theorem}\label{sec:realization}

We first build a vertex-critical graph of any sufficiently large order with the parity needed by the odd frame.  We call a graph $H$ \emph{vertex-$q$-critical} if $\chi(H)=q$ and $\chi(H-v)=q-1$ for every $v\in V(H)$.

\begin{lemma}[Vertex-critical padding]\label{lem:critical-padding}
Let $k\ge3$ and let $r\ge k$ satisfy $r\equiv k\pmod2$.  Define
\[
A_{k,r}=
\begin{cases}
C_r,&k=3,\\
K_{k-3}\vee C_{r-k+3},&k\ge4.
\end{cases}
\]
Then $A_{k,r}$ is a vertex-$k$-critical graph on $r$ vertices, and $\omega(A_{k,r})\le k$.
\end{lemma}

\begin{proof}
If $k=3$, then $r$ is odd and $A_{3,r}=C_r$ is vertex-$3$-critical, with clique number at most three.  Suppose that $k\ge4$.  The cycle $C_{r-k+3}$ has odd length at least three and is vertex-$3$-critical, while $K_{k-3}$ is vertex-$(k-3)$-critical.  The join of a vertex-$p$-critical graph and a vertex-$q$-critical graph is vertex-$(p+q)$-critical: chromatic numbers add under joins, and deleting a vertex lowers the chromatic number in the factor containing it.  Hence $A_{k,r}$ is vertex-$k$-critical.  Its clique number is at most $(k-3)+3=k$.
\end{proof}

We now state the realization theorem in its full form.

\begin{theorem}[Terminal-trace realization and parameter preservation]\label{thm:realization}
Fix $k\ge3$.  Let $\calH$ be a finite nonempty clutter, set
\[
T=V(\calH):=\bigcup_{e\in E(\calH)}e,
\]
and let $M\ge1$.  There is a $2$-connected graph $G=G(k,\calH,M)$ containing $T$ as an independent set such that the following hold.

\begin{enumerate}[label=\textup{(\roman*)},leftmargin=*]
\item \label{item:chromatic-host}
$G$ is $2$-connected, $\chi(G)=k+1$, $\omega(G)=k$, and $|E(G)|\le k|V(G)|$.

\item \label{item:trace-colouring}
For every $X\subseteq V(G)$, if $X\cap T$ contains no edge of $\calH$, then $G[X]$ is $k$-colourable.

\item \label{item:designated-cycles}
For every $e\in E(\calH)$ there are $M$ distinct $k$-bad odd cycles
\[
C_{e,1},\ldots,C_{e,M}
\]
with $\tr_T(C_{e,j})=e$.  The sets $V(C_{e,j})\setminus T$, over all pairs $(e,j)$, are pairwise disjoint.
\end{enumerate}

Consequently,
\[
E(\calH)=\min_{\subseteq}\tr_T(\calC_k(G)),
\]
where $\min_{\subseteq}$ denotes the inclusion-minimal members.  Moreover,
\[
\nu_k^*(G)=\nu^*(\calH).
\]
If $M\ge\tau(\calH)$, then
\[
\tau_k(G)=\tau(\calH),
\]
and, for every integer $b\ge1$, if $M\ge b$, then
\[
\nu_k^{1/b}(G)=\nu_b^{\mathbb Z}(\calH).
\]
In particular, if $B\ge1$ is prescribed and $M\ge\max\{\tau(\calH),B\}$, then the same graph $G$ satisfies
\[
\tau_k(G)=\tau(\calH),\qquad
\nu_k^*(G)=\nu^*(\calH),\qquad
\nu_k^{1/c}(G)=\nu_c^{\mathbb Z}(\calH)\quad(1\le c\le B).
\]
\end{theorem}

\begin{proof}
Choose an integer
\[
r\ge \max\{k,\max_{e\in E(\calH)}|e|\}
\qquad\text{with}\qquad r\equiv k\pmod2,
\]
and put $A=A_{k,r}$.

For each pair $(e,j)\in E(\calH)\times[M]$, take a private set $P_{e,j}$ of size $r-|e|$ and put
\[
U_{e,j}=e\cup P_{e,j}.
\]
Thus $|U_{e,j}|=r$, and only the vertices of $e$ are shared with other gadgets.  Take a private copy $A_{e,j}$ of $A$, whose vertices are denoted by $a_u$ for $u\in U_{e,j}$, and a private clique $Q_{e,j}\cong K_{k-1}$.  Add all edges between $U_{e,j}$ and $Q_{e,j}$, and add the matching edges
\[
ua_u\qquad(u\in U_{e,j}).
\]
Let $F_{e,j}$ be the resulting graph on
\[
U_{e,j}\cup V(A_{e,j})\cup V(Q_{e,j}).
\]

We first note that
\[
\chi(F_{e,j})=k+1. \tag{3.1}\label{eq:core-chromatic}
\]
Indeed, in a hypothetical $k$-colouring, the clique $Q_{e,j}$ uses $k-1$ colours.  Every vertex of $U_{e,j}$ is adjacent to all of $Q_{e,j}$, so all vertices of $U_{e,j}$ are forced to use the one remaining colour.  The matching edges $ua_u$ then forbid that colour on every vertex of $A_{e,j}$, giving a $(k-1)$-colouring of the $k$-chromatic graph $A$, a contradiction.  Conversely, colour $Q_{e,j}$ with colours $1,\ldots,k-1$, colour $U_{e,j}$ with colour $k$, and colour $A_{e,j}$ with the $k$ colours $1,\ldots,k-1,k+1$.

We place a private odd frame through all vertices of $F_{e,j}$.  Order these vertices as
\[
z_1,z_2,\ldots,z_m.
\]
For $1\le i<m$, add a private path of length two from $z_i$ to $z_{i+1}$, and add a private path of length three from $z_m$ to $z_1$.  All internal vertices of these paths are new.  Their union is an odd cycle $C_{e,j}$ of length $2m+1$ containing every vertex of $F_{e,j}$.  Let $D_{e,j}$ be the graph formed by $F_{e,j}$ and this private frame.  To make the realization $2$-connected without changing any designated cycle, take a new cycle $R$ of length at least four.  For every $(e,j)$, choose distinct internal vertices $x_{e,j},y_{e,j}$ of the private frame paths (and hence vertices private to this gadget) and distinct vertices $x'_{e,j},y'_{e,j}$ of $R$, and join $x_{e,j}$ to $x'_{e,j}$ and $y_{e,j}$ to $y'_{e,j}$ by two internally vertex-disjoint paths of length two.  All internal vertices of these attachment paths are new, and $R$ is chosen long enough that all attachment vertices on it may be distinct.  Let $G$ be the resulting graph, with the terminal set $T$ shared and every other gadget vertex private.  Figure~\ref{fig:realization-gadget} summarizes one gadget and its two attachments to the central cycle.

\begin{figure}[t]
\centering
\resizebox{\linewidth}{!}{%
\begin{tikzpicture}[
  box/.style={draw,rounded corners,align=center,minimum height=10mm,inner sep=5pt},
  every node/.style={font=\small}
]
\node[box,minimum width=34mm] (U) at (0,0)
  {$U_{e,j}=e\cup P_{e,j}$\\[-1pt]\scriptsize terminals and private padding};
\node[box,minimum width=28mm] (Q) at (0,2.05)
  {$Q_{e,j}\cong K_{k-1}$};
\node[box,minimum width=31mm] (A) at (4.35,0)
  {$A_{e,j}\cong A_{k,r}$};

\draw[very thick] (Q.south) --
  node[right,fill=white,inner sep=1pt,font=\scriptsize]{complete join} (U.north);
\draw[dashed,very thick] (U.east) --
  node[above,fill=white,inner sep=1pt,font=\scriptsize]{matching $ua_u$} (A.west);

\node[draw,densely dotted,rounded corners,fit=(U)(Q)(A),inner sep=7mm,
      label={[font=\small]below:$F_{e,j}$}] (F) {};
\node[draw,rounded corners=18pt,fit=(F),inner xsep=9mm,inner ysep=7mm,
      label={[font=\small]above:private odd frame $C_{e,j}$}] (frame) {};

\node[draw,ellipse,minimum width=30mm,minimum height=13mm,right=38mm of frame] (R)
  {central cycle $R$};
\coordinate (x) at ($(frame.east)+(0,0.65)$);
\coordinate (y) at ($(frame.east)+(0,-0.65)$);
\coordinate (xp) at ($(R.west)+(0,0.30)$);
\coordinate (yp) at ($(R.west)+(0,-0.30)$);
\node[circle,fill,inner sep=1.4pt] at (x) {};
\node[circle,fill,inner sep=1.4pt] at (y) {};
\node[circle,fill,inner sep=1.4pt] at (xp) {};
\node[circle,fill,inner sep=1.4pt] at (yp) {};
\draw[thick] (x) -- (xp);
\draw[thick] (y) -- (yp);
\node[circle,draw,fill=white,inner sep=1.2pt] at ($(x)!0.5!(xp)$) {};
\node[circle,draw,fill=white,inner sep=1.2pt] at ($(y)!0.5!(yp)$) {};
\node[align=center,font=\scriptsize,below=4mm of R]
  {two internally disjoint\\length-two attachment paths};
\end{tikzpicture}%
}
\caption{Schematic of one realization gadget.  Individual edges inside the complete join, the critical graph, and the odd frame are suppressed.  Only the terminals in $e\subseteq U_{e,j}$ are shared with other gadgets.}
\label{fig:realization-gadget}
\end{figure}
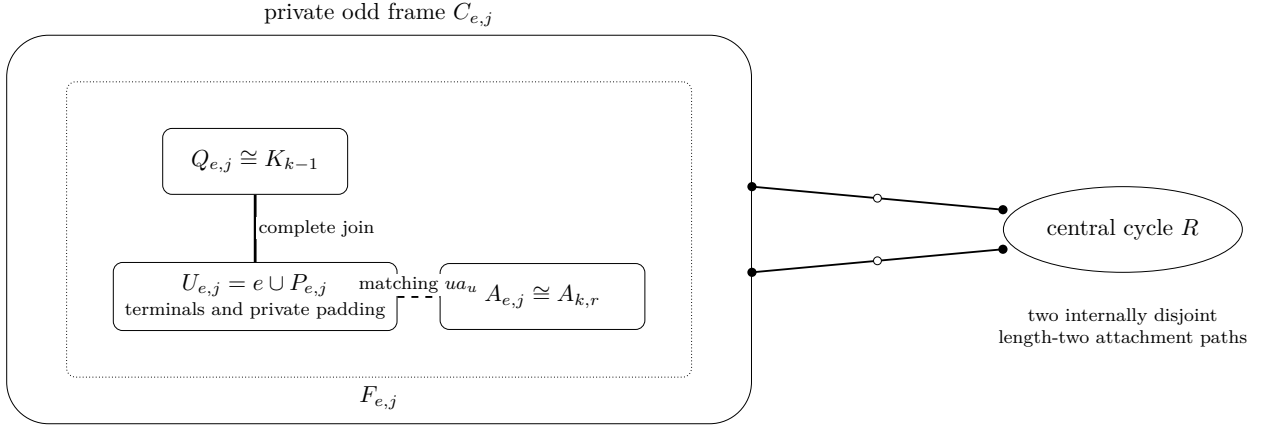

The displayed $(k+1)$-colouring of each $F_{e,j}$ is consistent across gadgets because every terminal receives colour $k$.  A colouring of the two endpoints of a path of length two or three extends over the path whenever at least three colours are available.  Colour $R$ properly and extend over the attachment paths.  Hence the colourings extend to a $(k+1)$-colouring of $G$.  Equation~\eqref{eq:core-chromatic} gives the reverse inequality, proving $\chi(G)=k+1$.

A clique in $F_{e,j}$ contains at most one vertex of $U_{e,j}$, no vertex of both $Q_{e,j}$ and $A_{e,j}$, and cannot use a matching edge $ua_u$ together with a third core vertex.  Thus its size is at most $k$, while $Q_{e,j}\cup\{u\}$ is a $K_k$ for every $u\in U_{e,j}$.  We also rule out cliques spanning several gadgets.  The terminal set $T$ is independent, and private vertices belonging to different gadgets are never adjacent.  Hence a clique cannot contain private vertices from two different gadgets, and every clique of size at least three that meets a core or frame is contained in one local gadget configuration.  An internal frame vertex has degree two inside its private frame and cannot enlarge a clique beyond size three.  The central cycle and the subdivided attachment paths create no clique of size three.  Since $k\ge3$, this proves $\omega(G)=k$.

We verify $2$-connectivity explicitly.  Each $D_{e,j}$ contains the spanning cycle $C_{e,j}$, so $D_{e,j}-v$ is connected whenever $v\in V(D_{e,j})$; similarly, $R-v$ is connected whenever $v\in V(R)$.  Fix $v\in V(G)$.

If $v\in V(R)$, then for each gadget at most one of its two attachment paths loses its endpoint on $R$, because all attachment vertices on $R$ are distinct.  The other attachment remains, and the internal vertex left on a broken attachment path stays connected to the gadget side, hence through the surviving attachment to $R-v$.  If $v$ is an internal vertex of an attachment path, the two remaining ends of that broken path lie in the gadget and in $R$, respectively, while the second attachment still joins those two connected sides.  If $v$ is a nonterminal vertex of some $D_{e,j}$, then $D_{e,j}-v$ is connected; deleting $v$ can destroy at most one attachment because the two private attachment endpoints $x_{e,j},y_{e,j}$ are distinct, and any remaining internal vertex of the destroyed path stays attached to $R$.  Finally, if $v\in T$, then every gadget containing $v$ remains connected after deleting $v$, and both of its attachments survive because their endpoints were chosen among private internal frame vertices.  Gadgets not containing $v$ are unaffected.

In every case, all remaining gadget vertices lie in a connected subgraph meeting $R$ or $R-v$, and hence $G-v$ is connected.  Thus $G$ is $2$-connected.

The realization is also sparse.  Put $d=M|E(\calH)|$ and recall that
\[
m=|V(F_{e,j})|=2r+k-1.
\]
The edge count of the core is obtained as follows.  First,
\[
\begin{aligned}
|E(A_{e,j})|
&=\binom{k-3}{2}+(r-k+3)+(k-3)(r-k+3)\\
&=(k-2)r-\frac{k(k-3)}{2},
\end{aligned}
\]
where for $k=3$ the first and third terms vanish and $A_{3,r}=C_r$.  Therefore
\[
\begin{aligned}
|E(F_{e,j})|
&=|E(A_{e,j})|+\binom{k-1}{2}+(k-1)r+r\\
&=(2k-2)r+1.
\end{aligned}
\]
Its frame and two attachment paths add $4r+2k+3$ edges.  There are $2d$ attachment vertices on the central cycle, so $R$ may be chosen with length $\max\{4,2d\}\le4d$.  Hence
\[
|E(G)|\le d\bigl((2k+2)r+2k+8\bigr).
\]
For a gadget indexed by $(e,j)$, the exact number of private vertices is
\[
(r-|e|)+r+(k-1)+(m+1)+2
   =4r-|e|+2k+1
   \ge 3r+2k+1.
\]
Here the five terms count, respectively, the private vertices of $U_{e,j}$, the copy of $A$, the clique $Q_{e,j}$, the internal frame vertices, and the two attachment-path interiors.  Therefore
\[
|V(G)|\ge d(3r+2k+1).
\]
Since $k\ge3$ and $r\ge k$,
\[
(2k+2)r+2k+8\le k(3r+2k+1),
\]
and therefore $|E(G)|\le k|V(G)|$.  This completes~\ref{item:chromatic-host}.

We prove the stronger colouring assertion~\ref{item:trace-colouring}.  Let $X\subseteq V(G)$ and suppose that $X\cap T$ contains no edge of $\calH$.  Colour every vertex of $X\cap T$ with colour $k$.  Consider a gadget indexed by $(e,j)$.  Choose
\[
t_e\in e\setminus X,
\]
which is possible because $e\nsubseteq X\cap T$.  Colour every present vertex of $U_{e,j}$ with colour $k$ and colour the present vertices of $Q_{e,j}$ with their fixed colours in $[k-1]$.  Since $A$ is vertex-$k$-critical, $A-a_{t_e}$ has a $(k-1)$-colouring.  Use such a colouring on the present vertices of $A_{e,j}-a_{t_e}$; if $a_{t_e}\in X$, give it colour $k$.  The only matching edge incident with $a_{t_e}$ has other end $t_e\notin X$, and every other present matching edge joins colour $k$ to a colour in $[k-1]$.  This properly colours the core induced by $X$.

The intersection of $X$ with each private frame path is a disjoint union of paths, and in each component at most the two ends of the original frame path have already been coloured.  On a frame path of length two, a present internal vertex can be assigned a colour different from the colours of its two present neighbours.  On a frame path of length three, colour the two present internal vertices successively; at each step at most two colours are forbidden.  Missing boundary or internal vertices only make the extension easier.  Since $k\ge3$, the core colouring therefore extends independently over every present part of every frame path.  Next colour $R[X\cap V(R)]$, which is a subgraph of a cycle and hence is $3$-colourable, and apply the same length-two extension to every present part of an attachment path.  The resulting colourings agree on shared terminals and give a $k$-colouring of $G[X]$.

Each $C_{e,j}$ is odd and its span contains $F_{e,j}$, so by~\eqref{eq:core-chromatic} it is $k$-bad.  Its only terminals are the vertices of $e$, and all nonterminal vertices are private to $(e,j)$.  This proves~\ref{item:designated-cycles}.

If $C$ is any $k$-bad odd cycle, then $G[V(C)]$ is not $k$-colourable.  By~\ref{item:trace-colouring}, its trace contains an edge of $\calH$.  Conversely, each $e\in E(\calH)$ is the trace of a designated bad cycle.  Since $\calH$ is a clutter, the inclusion-minimal traces are exactly its edges.

It remains to prove parameter preservation.  Let $Z\subseteq T$ be a transversal of $\calH$.  Every bad cycle has a trace containing an edge of $\calH$, and hence meets $Z$.  Therefore
\[
\tau_k(G)\le\tau(\calH).
\]
For the reverse inequality, suppose $M\ge\tau(\calH)$ and let $X\subseteq V(G)$ have size smaller than $\tau(\calH)$.  Then $X\cap T$ misses some edge $e\in E(\calH)$.  The $M$ cycles $C_{e,1},\ldots,C_{e,M}$ avoid $X\cap T$ and have pairwise disjoint nonterminal parts.  Meeting all of them requires at least $M\ge\tau(\calH)$ vertices, so $X$ is not a transversal.  Hence $\tau_k(G)=\tau(\calH)$.

Now fix $b\le M$.  Given a $1/b$-integral packing of bad cycles, choose for each packed cycle one edge of $\calH$ contained in its terminal trace.  If an edge $e$ is chosen $z_e$ times, then
\[
\sum_{e\ni t}z_e\le b\qquad(t\in T),
\]
so $(z_e)$ is a feasible integer $b$-packing of $\calH$.  Thus
\[
\nu_k^{1/b}(G)\le\nu_b^{\mathbb Z}(\calH).
\]
Conversely, let $(z_e)$ be a feasible integer $b$-packing.  Since every edge is nonempty, $z_e\le b\le M$.  For each $e$, choose $z_e$ distinct designated cycles among $C_{e,1},\ldots,C_{e,M}$.  Terminal loads are exactly the hypergraph loads and every nonterminal vertex is used at most once.  This gives a $1/b$-integral packing of the same cardinality, proving equality.

Finally, map any fractional packing of bad cycles to $\calH$ by choosing one contained trace edge for each cycle and summing the corresponding weights.  Terminal loads show that the resulting hyperedge weights form a feasible fractional packing, so
\[
\nu_k^*(G)\le\nu^*(\calH).
\]
Conversely, let $(y_e)$ be a fractional packing of $\calH$.  Put weight $y_e$ on the single designated cycle $C_{e,1}$.  Terminal loads are feasible, and $y_e\le1$ because $e$ is nonempty, so every private vertex also has load at most one.  Hence
\[
\nu_k^*(G)\ge\nu^*(\calH),
\]
completing the proof.
\end{proof}

\begin{remark}[The reusable mechanism]\label{rem:mechanism}
The proof separates into two independent devices.  The \emph{colour-forcing core} turns a terminal set $e$ into a $(k+1)$-chromatic obstruction while deleting any one terminal from $e$ restores a $k$-colouring.  The \emph{odd frame} converts that obstruction into a cycle without changing its terminal trace.  The converse colouring in Theorem~\ref{thm:realization}\ref{item:trace-colouring} upgrades the designated cycles to an exact realization of the complete bad-cycle trace system.
\end{remark}

\begin{remark}[Size and exact reductions]
Let $r$ be the padded rank used in the proof.  The construction has $O_k(M|E(\calH)|r)$ vertices and is polynomial in the explicit incidence description of $\calH$.  Taking $M\ge |T|$ therefore gives an exact optimum-preserving reduction from \textsc{Hitting Set} to the transversal problem for $k$-bad odd cycles on $2$-connected $(k+1)$-chromatic $K_{k+1}$-free graphs.  If a congestion bound $b$ is also prescribed, taking $M\ge\max\{|T|,b\}$ preserves its integer packing optimum and the fractional optimum at the same time.
\end{remark}

\section{Sharp global consequences}\label{sec:sharp}

Let $\binom{[n]}r$ denote the complete $r$-uniform hypergraph on $[n]$.

\begin{lemma}[Complete uniform hypergraphs]\label{lem:complete-hypergraph}
Let $1\le r\le n$, and put $\calH=\binom{[n]}r$.  Then
\[
\tau(\calH)=n-r+1
\qquad\text{and}\qquad
\nu^*(\calH)=\frac nr.
\]
Moreover, if $c\ge1$ and
\[
r>c(n-r), \tag{4.1}\label{eq:intersection-condition}
\]
then
\[
\nu_c^{\mathbb Z}(\calH)=c.
\]
\end{lemma}

\begin{proof}
A set meets every $r$-subset of $[n]$ exactly when its complement has size at most $r-1$, giving $\tau(\calH)=n-r+1$.

For every fractional packing $(y_e)$,
\[
r\sum_e y_e
 =\sum_{v\in[n]}\sum_{e\ni v}y_e
 \le n,
\]
so $\nu^*(\calH)\le n/r$.  Equality is attained by assigning weight
$\binom{n-1}{r-1}^{-1}$ to every edge.

The lower bound $\nu_c^{\mathbb Z}(\calH)\ge c$ follows by assigning multiplicity $c$ to one edge.  Suppose that $c+1$ edge occurrences are selected.  Their common intersection has size at least
\[
n-(c+1)(n-r)=(c+1)r-cn=r-c(n-r)>0.
\]
A vertex in the common intersection has load at least $c+1$, contradicting feasibility.  Hence $\nu_c^{\mathbb Z}(\calH)=c$.
\end{proof}

\begin{theorem}[Sharp simultaneous failure]\label{thm:sharp-failure}
Fix $k\ge3$.  For all integers $b,N\ge1$ and every $\varepsilon>0$, there is a $2$-connected graph $G$ such that
\[
\chi(G)=k+1,\qquad \omega(G)=k,\qquad |E(G)|\le k|V(G)|,
\]
and
\[
\tau_k(G)=N,
\qquad
\nu_k^{1/c}(G)=c\quad(1\le c\le b),
\qquad
1<\nu_k^*(G)<1+\varepsilon
\]
whenever $N\ge2$.  For $N=1$, the same statement holds with $\nu_k^*(G)=1$.
\end{theorem}

\begin{proof}
Choose an integer $r\equiv k\pmod2$ such that
\[
r>\max\left\{b(N-1),\frac{N-1}{\varepsilon},k-1\right\},
\]
and put
\[
n=r+N-1,
\qquad
\calH=\binom{[n]}r.
\]
Choose $M\ge\max\{b,N\}$ and apply Theorem~\ref{thm:realization}.  Lemma~\ref{lem:complete-hypergraph} gives
\[
\tau_k(G)=n-r+1=N
\]
and, for every $1\le c\le b$,
\[
\nu_k^{1/c}(G)=\nu_c^{\mathbb Z}(\calH)=c.
\]
It also gives
\[
\nu_k^*(G)=\frac nr=1+\frac{N-1}{r}<1+\varepsilon.
\]
This value is larger than one exactly when $N\ge2$.

The graph is $2$-connected by Theorem~\ref{thm:realization}.
\end{proof}

\begin{corollary}[No bounded-congestion or fractional Erd\H{o}s--P\'{o}sa theorem]\label{cor:no-ep}
For every fixed $k\ge3$ and every finite $b\ge1$, the family of $k$-bad odd cycles has no $1/b$-integral Erd\H{o}s--P\'{o}sa property, even on $2$-connected $(k+1)$-chromatic $K_{k+1}$-free graphs of average degree at most $2k$.  It also has no fractional Erd\H{o}s--P\'{o}sa property on this class.
\end{corollary}

\begin{proof}
For the $1/b$-integral statement, apply Theorem~\ref{thm:sharp-failure} with the prescribed value of $b$ and let $N$ tend to infinity.  Then $\nu_k^{1/b}(G)=b$ while $\tau_k(G)=N$.  For the fractional statement, fix $\varepsilon=1$ and take $N\ge2$.  The same theorem gives
\[
\lceil\nu_k^*(G)\rceil=2
\qquad\text{and}\qquad
\tau_k(G)=N,
\]
with $N$ arbitrary.  Both conclusions contradict the corresponding definitions from Section~\ref{sec:parameters}.
\end{proof}

The fractional constant in Theorem~\ref{thm:sharp-failure} cannot be replaced by one when the transversal number exceeds one.

\begin{proposition}[The fractional obstruction is optimal]\label{prop:fractional-optimal}
Let $\calF$ be a nonempty finite family of nonempty sets.  If $\tau(\calF)\ge2$, then $\nu^*(\calF)>1$.
\end{proposition}

\begin{proof}
Certainly $\nu^*(\calF)\ge1$, by putting weight one on a single member.  Suppose equality holds, and let $(x_v)$ be a minimum fractional transversal of total weight one.  For every $F\in\calF$,
\[
1\le\sum_{v\in F}x_v\le\sum_v x_v=1.
\]
Hence every vertex with positive weight belongs to every member of $\calF$.  Any such vertex is an integral transversal of size one, a contradiction.
\end{proof}

\begin{remark}[Exact preservation of Hitting Set]
The equality $\tau_k(G)=\tau(\calH)$ embeds the transversal problem for arbitrary finite clutters exactly into the bad-cycle transversal problem while keeping $\chi(G)=k+1$ and $\omega(G)=k$.  Consequently, the bad-cycle transversal problem inherits the full obstruction theory of finite hypergraph transversals.
\end{remark}

\section{The shortest bad cycles}\label{sec:shortest}

The global universality of Section~\ref{sec:realization} begins only after the first length layer.  At minimum length, one fixed graph controls the entire family.

For $k\ge2$, define
\[
\ell_k=
\begin{cases}
 k+1,&k\text{ even},\\
 k+2,&k\text{ odd}.
\end{cases}
\]
Define $W_k$ as follows.  If $k$ is even, let $W_k=K_{k+1}$.  If $k$ is odd, start with a clique $Q\cong K_{k+1}$ and add one vertex $x$ adjacent to two vertices of $Q$.  Copies are not required to be induced.

\begin{theorem}[Shortest-witness theorem]\label{thm:shortest-witness}
Let $k\ge2$.  An odd cycle of minimum possible bad length $\ell_k$ is $k$-bad if and only if its vertex span contains a copy of $W_k$.  Equivalently, a graph $G$ contains a $k$-bad odd cycle of length $\ell_k$ if and only if $W_k\subseteq G$.
\end{theorem}

\begin{proof}
Every $k$-bad cycle has at least $k+1$ vertices, and its length is odd, so its length is at least $\ell_k$.

The graph $W_k$ has chromatic number $k+1$ and a spanning odd cycle.  This is immediate for even $k$.  For odd $k$, take a Hamilton path in the clique $Q$ between the two neighbours of $x$ and close it through $x$.  Thus a copy of $W_k$ gives a bad cycle of length $\ell_k$.

Conversely, let $C$ be a $k$-bad cycle of length $\ell_k$, and put $H=G[V(C)]$.  If $k$ is even, then $H$ has $k+1$ vertices and chromatic number at least $k+1$, so $H=K_{k+1}=W_k$.

Suppose that $k$ is odd.  Then $H$ has $k+2$ vertices and $\chi(H)\ge k+1$.  A $k$-colouring of $H$ is equivalent to a partition of $V(\overline H)$ into at most $k$ cliques.  Since $|V(H)|=k+2$, such a partition must save at least two classes relative to the partition into singletons.  This is possible exactly when either one clique class has size at least three, yielding a triangle in $\overline H$, or two distinct clique classes each have size at least two, yielding two vertex-disjoint edges in $\overline H$.  Conversely, either a triangle or two vertex-disjoint edges gives a clique partition with at most $k$ classes.  Since $\chi(H)\ge k+1$, no such $k$-colouring exists.  Hence $\overline H$ is triangle-free and has matching number at most one.

If $E(\overline H)=\emptyset$, then $H=K_{k+2}$, which contains $W_k$, and the conclusion is immediate.  Otherwise every triangle-free graph with matching number at most one is a star together with isolated vertices.  Indeed, choose an edge $ab$.  If some edge is not incident with $a$, it must meet $ab$, say it is $bc$; then every further edge must meet both $ab$ and $bc$, and triangle-freeness forces it to be incident with $b$.  Thus all nonedges of $H$ share one endpoint $x$, and the other $k+1$ vertices form a clique $Q$.  Since $C$ is spanning, $x$ has at least two neighbours on $C$, both in $Q$.  These vertices together with $Q$ form a copy of $W_k$.
\end{proof}

\begin{corollary}[Universality beyond the shortest layer]\label{cor:universality-beyond-shortest}
Let $k\ge3$, and let $G$ be any graph produced by Theorem~\ref{thm:realization}.  Then $G$ contains no $k$-bad odd cycle of length $\ell_k$.
\end{corollary}

\begin{proof}
Every copy of $W_k$ contains a $K_{k+1}$, whereas Theorem~\ref{thm:realization} gives $\omega(G)=k$.  The conclusion follows from Theorem~\ref{thm:shortest-witness}.
\end{proof}

\begin{corollary}[Finite first layer]\label{cor:finite-first-layer}
For fixed $k$, hitting or deleting all shortest $k$-bad odd cycles is exactly the corresponding fixed-subgraph problem for $W_k$.
\end{corollary}

\begin{remark}
Theorem~\ref{thm:shortest-witness} is the local counterpart to Theorem~\ref{thm:realization}.  The first length layer is governed by one finite witness, while unrestricted lengths can simulate every finite clutter.
\end{remark}

\section{Further directions: bounded certificates and moving cores}\label{sec:bounded}

Sections~\ref{sec:realization}--\ref{sec:shortest} form the main theorem chain.  We close with a brief structural outlook explaining what remains when the chromatic obstruction on a bad cycle has bounded order.  Throughout this section, $k\ge2$.

\begin{definition}[Bounded chromatic certificates]
Let $C$ be an odd cycle of a graph $G$.  A set $Y\subseteq V(C)$ is a \emph{$k$-chromatic certificate} if $\chi(G[Y])\ge k+1$.  The cycle has \emph{certificate size at most $s$} if it has such a certificate with $|Y|\le s$.
\end{definition}

A \emph{real copy} of a fixed labelled graph $B$ in $G$ is an injective edge-preserving map $\phi:V(B)\to V(G)$.  An odd cycle is \emph{$B$-certified} if its span contains a real copy of $B$.  Write $\calC_B(G)$ for the family of $B$-certified odd cycles.

\begin{lemma}[Vertex-critical-core reduction]\label{lem:critical-core}
Let $Y\subseteq V(G)$ satisfy $\chi(G[Y])\ge k+1$, and choose $S\subseteq Y$ inclusion-minimal with this property.  Then $B=G[S]$ is vertex-$(k+1)$-critical and therefore $2$-connected.
\end{lemma}

\begin{proof}
Every proper induced subgraph of $B$ is $k$-colourable.  If $\chi(B)\ge k+2$, then colouring $B-v$ with at most $k$ colours and assigning a fresh colour to $v$ would give $\chi(B)\le k+1$, a contradiction.  Hence $\chi(B)=k+1$.  For every $v\in V(B)$, minimality gives $\chi(B-v)\le k$; if $\chi(B-v)\le k-1$, one fresh colour on $v$ would give $\chi(B)\le k$.  Thus $\chi(B-v)=k$ for every $v$.  A vertex-critical graph of chromatic number at least three is connected and has no cutvertex: colour the blocks at a cutvertex separately and permute colours so that the cutvertex receives the same colour.  Hence $B$ is $2$-connected.
\end{proof}

Fix a labelled graph $B$ with vertex set $[r]$.  For a cyclic order $\pi=(\pi_1,\ldots,\pi_r)$ and a vector $p\in\{0,1\}^r$ with odd coordinate sum, call a cycle a \emph{$(B,\pi,p)$-cycle} if it contains a real copy $\phi$ of $B$, the vertices $\phi(\pi_1),\ldots,\phi(\pi_r)$ occur in this cyclic order, and the arc between consecutive core vertices has parity prescribed by $p$.

\begin{proposition}[Finite real-core parity normal form]\label{prop:parity-normal-form}
For fixed $k$ and $s$, the odd cycles having certificate size at most $s$ form a finite union of $(B,\pi,p)$-cycle families, where $B$ ranges over the labelled $2$-connected vertex-$(k+1)$-critical graphs on at most $s$ vertices.
\end{proposition}

\begin{proof}
Choose an inclusion-minimal certificate and apply Lemma~\ref{lem:critical-core}.  Reading its vertices around the witnessing cycle determines $\pi$, while the intervening arc parities determine $p$ and have odd sum.  Conversely, every $(B,\pi,p)$-cycle is odd and contains the prescribed real core.  Only finitely many labelled cores, cyclic orders, and parity vectors occur for fixed $s$.
\end{proof}

Because the union is finite, pattern-wise packing-covering bounds combine by taking the union of the corresponding transversals.  Here a real copy keeps every edge of $B$ as an actual chord in the cycle span; only the arcs between consecutive core vertices are flexible.  Two elementary observations indicate the positive boundary.  If $B$ has an odd Hamilton cycle and $r=|V(B)|$, the rank-$r$ hypergraph of real $B$-copies gives
\[
\tau(\calC_B(G))\le r\,\nu(\calC_B(G))
\qquad\text{and}\qquad
\tau(\calC_B(G))\le r\,\nu^*(\calC_B(G)).
\]
If $B$ is Hamiltonian-connected, then every $B$-certified odd cycle contains a cycle formed by one Hamilton path in the core and one parity-correct external ear: read the core vertices around the cycle, choose an intervening arc whose parity agrees with $|V(B)|$, and replace the remaining arcs by a Hamilton path in the core.  The remaining difficulty is the simultaneous movement of arbitrarily many real copies.

For a fixed subgraph $Q\subseteq G$, let $\calC_Q^{\rm odd}(G)$ be the odd cycles containing every vertex of $Q$, and write $\tau_Q$, $\nu_Q^{1/2}$, and $\nu_Q^*$ for the corresponding parameters.

\begin{theorem}[Fixed real cores]\label{thm:fixed-core}
For every integer $s\ge1$ there is a nondecreasing function $f_s$ such that, whenever $|V(Q)|\le s$,
\[
\tau_Q(G)\le f_s\bigl(\nu_Q^{1/2}(G)\bigr)
\qquad\text{and}\qquad
\tau_Q(G)\le f_s\bigl(\lceil2\nu_Q^*(G)\rceil\bigr).
\]
\end{theorem}

\begin{proof}
Let $F_{m,1}$ be the bound supplied by the unified half-integral theorem of Gollin, Hendrey, Kawarabayashi, Kwon, and Oum~\cite{GollinEtAlHalf}, and set
\[
f_s(t)=\max\{F_{m,1}(q+1):1\le m\le s+1,\ 0\le q\le t\}.
\]
Use one $\mathbb Z_2$ coordinate in which every edge has label $1$, so a cycle has nonzero value exactly when it is odd.  For each $q\in V(Q)$, use one $\mathbb Z_3$ coordinate in which precisely the edges incident with $q$ have label $1$.  A cycle has value $2$ in this coordinate exactly when it contains $q$, and value $0$ otherwise.  Thus $\calC_Q^{\rm odd}(G)$ is the family of cycles avoiding the forbidden value $0$ in $m=|V(Q)|+1\le s+1$ coordinates.  If $t=\nu_Q^{1/2}(G)$, the cited theorem gives $\tau_Q(G)\le F_{m,1}(t+1)\le f_s(t)$.  The fractional inequality follows from $\nu_Q^{1/2}(G)\le2\nu_Q^*(G)$ and monotonicity of $f_s$.
\end{proof}

For every fixed $m$, the same argument handles a union of families associated with at most $m$ specified copies $Q_1,\ldots,Q_m$: applying Theorem~\ref{thm:fixed-core} to each copy and taking the union of the transversals gives a bound depending only on $s$ and $m$.

Allowing the core copy to move already contains a natural demand-matching problem.  Marx and Wollan showed that hereditary demand classes containing arbitrary matchings do not have a uniform integral Erd\H{o}s--P\'{o}sa analogue for valid paths~\cite{MarxWollan}.  We are not aware of a corresponding half-integral bound independent of the size of the matching.

\begin{problem}[Exact-one-marker cycles]\label{prob:exact-one-marker}
Let
\[
D=\{s_it_i:i\in I\}\subseteq\binom{V(G)}2\setminus E(G)
\]
be a matching on $V(G)$, and let $\calX(G,D)$ be the cycles in $G\cup D$ that use exactly one edge of $D$.  Does $\calX(G,D)$ satisfy a half-integral Erd\H{o}s--P\'{o}sa theorem with a bound independent of $|I|$?  Does the same hold after prescribing the parity of the $G$-part?
\end{problem}

Deleting the unique demand edge identifies these cycles with prescribed-pair paths in $G$.  Existing $A$-path theorems, including the finite-abelian-group result of Kwon and Yoo, and the finite-group theorem for non-null $S$--$T$ paths of Chekan et al., cover several nearby fixed-label settings~\cite{ChudnovskyEtAl,BruhnHeinleinJoos,KwonYoo,ChekanEtAl}, but do not directly give a uniform bound when the set of pair identities is unbounded.  Proposition~\ref{prop:parity-normal-form} leads to the broader question.

\begin{problem}[Moving real-core parity linkages]\label{prob:moving-core}
Fix a graph $B$, a cyclic order $\pi$ of $V(B)$, and an odd parity vector $p$.  Do the $(B,\pi,p)$-cycles satisfy a half-integral Erd\H{o}s--P\'{o}sa theorem?  At least, do they satisfy a fractional one?
\end{problem}

Theorem~\ref{thm:realization} shows that growing cores already encode arbitrary clutters.  Fixing the abstract real core leaves a finite moving-core linkage problem.  The main theorem and the bounded-certificate outlook therefore mark a sharp boundary between universal finite packing-covering complexity and finitely many fixed-core parity patterns.

\end{document}